\documentclass[12pt]{amsart}
\usepackage{amssymb,latexsym}
\usepackage{enumerate}

\newtheorem{thm}{Theorem}
\newtheorem{cor}[thm]{Corollary}

\newtheorem{pro}[thm]{Proposition}

\theoremstyle{definition}
\newtheorem{defin}[thm]{Definition}
\newtheorem{rem}[thm]{Remark}
\newtheorem{exa}[thm]{Example}

\begin{document}

\title[Nikolskii's property]{Generalized Nikolskii's property and asymptotic exponent in Markov's inequality}



\author{Miros\l{}aw Baran   \and
        Agnieszka Kowalska 
}


\maketitle

\begin{abstract}
We introduce  an asymptotic Markov's exponent and show that it is equal to
Mar\-kov's exponent for a wide class of norms. However it is not
true for all norms in the space of polynomials, as it will be
presented in few examples. We shall prove an important inequality
$m(q)\geq m(E)$, where $q$ is a norm in $ \mathbb{P}(
\mathbb{C}^N)$ with Nikolskii's property related to $E$. As a
consequence we obtain a~lower bound for the optimal exponent in
Markov's inequality considered with the $L^p$ norms and other
norms possessing Nikolskii type property. 
\end{abstract}

\section{Introduction}
\label{intro}
By $ \mathbb{P}(\mathbb{K}^N)$ ($\mathbb{P}_d(\mathbb{K}^N)$,
respectively) we shall denote the vector space of all polynomials
of $N$ variables with coefficients in the field $\mathbb{K}$ (with
total degree $\leq d$). Let us recall the multivariate Markov's
inequality
\begin{defin}\label{MarkovSet}
A compact set $E\subset \mathbb{K}^N$ admit \emph{Markov's
inequality} if there exist constants $M,m>0$ such that for all
polynomials $P\in \mathbb{P}(\mathbb{K}^N)$ and $j\in\{1,2,\ldots,
N\}$
\begin{equation}\label{Markov}
\left\| \frac{\partial}{\partial x_j} P\right \|_E \le M \: (\deg
P)^{m} \: \|p\|_E \
\end{equation}
where $\|\cdot \|_E$ is the supremum norm on $E$.
\end{defin}
A compact set $E$ with the above property is called the \emph{Markov's set}.
It is a~generalization of the classical inequality proven by A. A. Markov in 1889,
 which gives such estimate on $[-1,1]$. The development of the theory of generalizations of this
 inequality is still continuing.  More information about the various generalizations of
  Markov's inequality can be found in \cite{Pl90},\cite{Milovanovic},\cite {Pl98},\cite{Govil},\cite{RahmanSchmeisser},\cite{Pl06}.
It is important to know more about the best exponent in this inequality for a given set $E$. The notion $m(E)$ called Markov's exponent was defined in \cite{BaranPlesniak}. For a Markov set $E$ it is $
m(E):=\inf\{ s\!>\!0 \ : E \mbox{ is Markov's set with exponent } s \}$.  If $E$ is not Markov's set, we put $m(E):=\infty$. It is known that $m(E)\ge 2$ in the real case and $m(E)\ge 1$ in the complex one. The surprising fact, proved in \cite{BBCM}, is that Markov's inequality does not have to fulfilled with Markov's exponent (see also \cite{Gon2014}).

Similarly we define Markov's exponent with respect to other norms, if $q$ is a norm on $ \mathbb{P}(\mathbb{K}^N)$ we can define Markov's exponent for the norm $q$ as
\[
m(q)=\inf\{ s\!>\!0  :\  \exists_{C>0} \forall_{P\in
\mathbb{P}(\mathbb{K}^N)} \forall_{1\leq j \leq N}
q\left(\frac{\partial}{\partial x_j} P\right) \le C \: (\deg
P)^{s} \: q(P)\}.
\]

The Markov type inequalities were also considered in $L^p$ norms
(cf. \cite{HSzT},\cite{Baran5},\cite{LBCSroka},\cite{BorweinErdelyi},\cite{Goe1979},\cite{Goe1990},\cite{Goe1994},\cite{Goe1998},\cite{Pierzchala4}). In this case a progression in research seems to be
slower except $L^2$ norms are
considered (cf. e.g. \cite{BottcherDorfler09},
\cite{BottcherDorfler10},\cite{BottcherDorfler11},\-\cite{Dorfler87},\cite{Dorfler02},
\cite{AleksovNikolovShadrin16},\cite{AleksovNikolov17}). In
particular, an
example of a compact set in $
\mathbb{R}^N$  with cusps for which Markov's exponent (with
respect to the Lebesgue measure) is calculated, is still out of reach.

 We can consider
Markov's inequality for any other norm. Then Markov's exponent can
even be equal to $0$.

\begin{exa}
For the norm $\|P\|=\sum_{k=0}^{\infty}|P^{(k)}(0)|$ we have
\[
\|P'\|=\sum_{k=1}^{\infty}|P^{(k)}(0)|\leq \|P\|.
\]
\end{exa}

However, if we have a spectral norm $q$ (it means for every polynomial $P\in \mathbb{P}(\mathbb{C}^N)$, $q(P^n)=(q(P))^n$) and Markov's inequality holds for this norm, then the exponent $m(q)$ has to be not less than $1$.
Indeed, let us consider polynomials $P_j(x)=x_j^n$, for $j\in\{1,2,\ldots,N\}$. Then
\[
\|x_j\|^n=\|P_j\|\geq \frac{1}{M^n(n!)^m}\left\|\frac{\partial
^nP_j}{\partial x_j^n}\right\| =
\frac{n!}{M^n(n!)^m}\|1\|=\frac{1}{M^n}(n!)^{1-m}.
\]
Hence
\[
\|x_j\|\geq \frac{1}{M}(n!)^{(1-m)\frac{1}{n}}.
\]
This inequality is possible only for $m\geq 1$.

\begin{rem} It was proved in \cite{BKMO},
 the Markov type condition
$$\left\| \frac{\partial}{\partial z_j}P\right\|_E\leq M(\deg P)^m\| P\|_E,\
j=1, \dots N,  P\in \mathbb{P}(\mathbb{C}^N)$$  with positive
constants $M$ and $m$ is equivalent to the inequality
$$\left\| \sum\limits_{j=1}^N\frac{\partial^{2l}P}{\partial z_j^{2l}}
\right\|_E\leq M'_l(\deg P)^{2lm}\| P\|_E,\ P\in
\mathbb{P}(\mathbb{C}^N)
$$ with some positive constant $M_l'$. Here $E\subset\mathbb{R}^N$ and $l\in\mathbb{Z}_+$ is fixed.
In particular, Markov's property with exponent $m$ is
equivalent to the bound $||\Delta P||_E\leq M'(\deg
P)^{2m}||P||_E$.
\end{rem}

\section{Nikolskii's property}\label{NikProp}
\begin{defin} Let $E$ be a compact subset of $\mathbb{C}^N$. A norm $q=||\cdot||$ on $ \mathbb{P}(\mathbb{C}^N)$ is {\it
$E$-admissible} or {\it has Nikolskii's property} if there exist
 constants: positive $A,B$ and nonnegative $a,b$ such that for every $P\in\mathbb{P}(\mathbb{C}^N)$  with $\deg P\geq 1$ we have
$$||P||_E \leq A(\deg P)^a ||P||\mbox{ and } ||P||\leq B(\deg P)^b ||P||_E.$$
\end{defin}

If $q=||\cdot ||$ is $E$-admissible then
$$||P||_E =\lim\limits_{s\rightarrow\infty}||P^s||^{1/s}.$$ Since the
supremum norm is the main example of spectral norm (see \cite{Zel}) we can generalize
the above definition.

\begin{defin}
A norm $q=||\cdot||$ on $\mathbb{P}(\mathbb{C}^N)$ is {\it spectral
admissible} or {\it has the generalized Nikolskii's property} if there
exist a spectral norm $||\cdot ||_{\sigma}$ and
 constants: positive $A,B$ and nonnegative $a,b$ such that for every $P\in\mathbb{P}(\mathbb{C}^N)$  with $\deg P\geq 1$ we have
$$||P||_{\sigma} \leq A(\deg P)^a ||P||\mbox{ and } ||P||\leq B(\deg P)^b ||P||_{\sigma}.$$\end{defin} The spectral norm is given by the formula
$$q_{\sigma}(P)=||P||_{\sigma}
=\lim\limits_{s\rightarrow\infty}||P^s||^{1/s}.$$
\bigskip
By way of illustration, here are examples of such norms.

\begin{exa} Let $E$ be a compact subset of $ \mathbb{C}$ and $r>0$ be fixed. Put (cf.
\cite{BKMO})
$$||P||=\sum\limits_{k=0}^\infty \frac{1}{k!}||P^{(k)}||_E\, r^k.$$
Then $$\lim\limits_{n\rightarrow\infty}||P^n||^{1/n}=\max\limits_{|\zeta
|\leq r}||P(x+\zeta )||_E.$$ Moreover for $||P||_\sigma:=\max\limits_{|\zeta
|\leq r}||P(x+\zeta )||_E$ we have
\[||P||_\sigma\leq ||P||\leq (\deg P+1)||P||_\sigma.\]
\end{exa}

\begin{exa} If $\mu $ is a probabilistic measure on $E$, then for $1\leq
s<\infty$ the norm

$$||P||=||P||_E+\left(\int\limits_E |P(z)|^sd\mu
(z)\right)^{1/s}$$ is $E$-admissible on $
\mathbb{P}(\mathbb{C}^N)$ with 
\[\lim\limits_{n\rightarrow\infty}||P^n||^{1/n} =\max
(||P||_E,{\rm ess}\sup\limits_E |P|)=||P||_E.\]
\end{exa}
\medskip

\begin{exa}
In the classical case of the interval $[-1,1]$ we have S.M.
Nikolskii's inequalities (cf. \cite{Nikolski},\cite{Timan},\cite{Milovanovic},\cite{Sroka})
\begin{align*}   \left(\frac{1}{2}\int\limits_{-1}^1|P(x)|^pdx\right)^{1/p}\leq &||P||_{[-1,1]}\\
\leq &(2(p+1)n^2)^{1/p}\left(
\frac{1}{2}\int\limits_{-1}^1|P(x)|^pdx\right)^{1/p}.\end{align*}
\end{exa}

\bigskip

\begin{exa}\label{Eop} (A generalization of Nikolskii's inequality) Let $\mu $ be a~pro\-babilistic measure on $E$ such that for a system of orthonormal
polynomials we have the inequality $||P||_E\leq B(\deg P)^\beta $ with some positive $\beta$, which is equivalent to the fact that for each polynomial $P$, $\deg P\geq 1,$
\begin{equation}\label{pEop}||P||_E\leq B_1(\deg P)^{\beta_1}||P||_2\end{equation} with some positive
constants $B_1,\beta_1$. Indeed, if $(P_\alpha )_{\alpha\in
\mathbb{N}^N}$ is an orthonormal system such that $\deg P_\alpha
=|\alpha |$ then for each polynomial $P$ with $\deg P\geq 1$, $||P||_E
\leq \binom{n+N}{n}\max\limits_{|\alpha |\leq n}|c_\alpha
|B|\alpha |^\beta$, where $c_\alpha=\int\limits_EP(z)
\overline{P_\alpha (z)}d\mu (z)$, so we  can take
$B_1=B\frac{2^N}{N!},\ \beta_1=\beta +N$.

Let us also note that the condition $||P||_E\leq B_1(\deg P
)^{\beta_1}||P||_2$ implies the inequality $||P||_E \leq B_1^{2/s}
(\lceil s\rceil )^{2\beta_1/s}(\deg P)^{\beta /s}||P||_{s}$,
$s\geq 1$. In particular, $||P||_E=||P||_\infty ={\rm
ess}\sup\limits_E|P|$.

Then for all $p\geq 1$ each norm $||P||_p=\left(\int\limits_E
|P(z)|^pd\mu (z)\right)^{1/p}$ is an $E$-admissible norm.
\end{exa}

\begin{rem} If $\mu$ is the normalized Lebesgue measure on a fat
compact set $E\subset\mathbb{R}^N$ then Nikolskii's inequality
implies Markov's property of $E$. It is a~consequence of main
results of \cite{Baran1}, \cite{Baran4} and \cite{Totik1} (cf.
\cite{Totik2,Totik3}) in one dimensional case.  Hence, if we want
to show that a given compact subset of $ \mathbb{R}^N$ possesses
Markov's property, it suffices to show Nikolskii's inequality
as in the example above. Generally, it is a very difficult task to
check Markov's property. Recently, a nontrivial result in this
topic has been obtained by R. Pierzcha\l{}a \cite{Pierzchala4}.
His remarkable result relates to a class of sets with a {\it
special parametric property} introduced by himself. This property
implies Nikolskii's inequality and thus Markov's property, as
it was noticed above (but it was not considered in
\cite{Pierzchala4}).

\end{rem}
\begin{exa} Let $E\subset \mathbb{R}^N$ and $\mu $ be a probabilistic
measure on $E$ with the following {\it density
condition}:
$$\exists G,\gamma >0\ \forall x\in E, r>0\ \ \ \mu (E\cap
\overline{B}(x,r))\geq Gr^\gamma .$$ Assuming $E$ has Markov's
property, one can prove (\ref{pEop}) for $E$. This method was used in the proof of Nikolskii's inequality in the
classical case (cf. \cite{Nikolski},\cite{Timan}) as well as in more
general situations investigated by A. Zeriahi \cite{Zer}, P.
Goetgheluck \cite{Goe1989} and A. Jonsson \cite{Jonsson2} (cf.
also
 \cite{Jonsson3}). Goetgheluck in \cite{Goe1989} proved
 that each UPC set in $
 \mathbb{R}^N$ (this wide family of sets was introduced by W.~Paw\l{}ucki and W.~Ple\'sniak in \cite{PaPl86}) satisfies the density condition and also by \cite{PaPl86}
 has Markov's property. Therefore each UPC set (in particular each
 compact fat subanalytic subset of $\mathbb{R}^N$, cf. \cite{PaPl86,PaPl88} for this deep result) satisfies the
 generalized Nikolskii's inequality with respect to the normalized
 Lebesgue measure $\mu$ and Markov's
 inequality in $L^p(\mu )$. However, no example is known of a set with cusp for which Markov's exponent (in $L^p(\mu ),\
 1\leq p<\infty$) is calculated.
\end{exa}

\begin{exa} Let $E\subset \mathbb{R}^N$. Put
\begin{equation}\label{NEi}
||P||=||P||_E+\int\limits_{int(E)}|D_jP(x)|dx.
\end{equation}
Since
$$
\int\limits_{int(E)}|D_jP(x)|dx\leq
\sqrt{N}\pi^N({\rm diam} (E))^{N-1}(\deg P)||P||_E,
$$
the norm $\|\cdot\|$ defined by (\ref{NEi}) is $E$-admissible. Ihe
last inequality follows from \cite{Baran1}, \cite{Baran4} and
\cite{Totik1}.
\end{exa}

\begin{exa}\label{ExSchur} Let $||P||=\sup\limits_{x\in
[-1,1]}|P(x)|\sqrt{1-x^2}$ be Schur's norm. Since \[||P||\leq
||P||_{[-1,1]}\leq (\deg P+1)||P||,\] Schur's norm is $[-1,1]$-admissible. Similarly, if we put \[||P||_\alpha =\sup\limits_{x\in
[-1,1]}|P(x)|(1-x^2)^{\alpha},\ \alpha >0,\] then (cf.
\cite{Baran5} for $\alpha \geq \frac{1}{2})$ the norm
$||\cdot||_\alpha$ is $[-1,1]$-admissible. Moreover,
if we replace the interval $[-1,1]$ by the unit closed ball $B:=\{ x\in
\mathbb{R}^N:\ ||x||_*\leq 1\}$ with respect to a fixed norm $||\cdot ||_{*}$ in $\mathbb{R}^N$ then the norm defined by \[||P||_\alpha =\sup\limits_{x\in
B}|P(x)|(1-||x||^2)^\alpha\] is $B$-admissible. A more general situation is contained
in the following way (cf. \cite{Baran3}). Let $\Omega$ be a bounded,
star-shaped (with respect to the origin) and symmetric domain in $
\mathbb{R}^N$ and let $E=\overline{\Omega}$. Let $v\in
S^{N-1}$ be a fixed direction, we assume that $\rho_v(tx)\geq M(1-|t|)^m,
t\in [-1,1], x\in \partial E$. Then for any $\alpha >0$, the norm
$||P||_\alpha =\sup\{ |P(tx)|(1-|t|)^\alpha:\ x\in\partial E,t\in
[-1,1]\}$ is $E$-admissible. Note that in this case $m(E,v)\leq
2m$ (cf. the first definition in the next section with
$D=v_1D_1+\dots+v_ND_N$).
\end{exa}

\begin{rem} The Schur inequality in Example \ref{ExSchur} is a special case
of the division type inequality, which is often called the Schur type
inequality. It was proved in \cite{LBC99} that on the complex
plane properties related to Markov's and Schur's inequalities are
equivalent.

\end{rem}

\begin{rem} If we have some norms with the generalized Nikolskii's property
(GNP), we can easily construct many other norms with
this property. For example, if $q_1,q_2$ have GNP (with spectral
norms $q_{1,\sigma},q_{2,\sigma}$) then
$q(P)=\left(q_1(P)^p+q_2(P)^p\right)^{1/p},\ 1\leq p\leq \infty$
has GNP with $q_\sigma=\max (q_{1,\sigma},q_{2,\sigma})$.
\end{rem}

\begin{rem} Let $||\cdot ||_0$ be a spectral norm in $ \mathbb{P}( \mathbb{C}^N)$ and let $||\cdot
||_1$ be a GNP norm with respect to $||\cdot ||_0$. If
$\alpha_j\in\mathbb{Z}_+^N$, $j=1,\dots ,l$  are fixed then we can
consider
$$||P||=||P||_0+\max\limits_{1\leq j\leq l}||D^{\alpha_j}P||_1.$$

We have $\lim\limits_{n\rightarrow\infty}||P^s||^{1/s}=||P||_0$
but GNP will be satisfied if and only if we have a
Markov-Nikolskii type bound
$$\max\limits_{1\leq j\leq l}||D^{\alpha_j}P||_1\leq C(\deg
P)^\gamma  ||P||_0.$$

Let us give two examples.

If $||P||=||P||_{[-1,1]\cup\{ 2\}}+||P'||_{[-1,1]\cup\{ 2\}}$ then
this norm does not satisfy GNP.

Now we define $||P||=||P||_E+||\frac{\partial P}{\partial x}||_{E}$, where
$E=\{ (x,y)\in\mathbb{R}^2:\ |x|<1,\ |y|\leq \exp
(-1/(1-|x|))\}\cup\{ (-1,0),(1,0)\}$. Since (cf. \cite{Baran3})
\[||\frac{\partial P}{\partial x}||_{E}\leq 2(\deg P)^2||P||_E,\]
the norm $||\cdot||$ possesses GNP.
\end{rem}

\section{Asymptotic exponent in Markov's inequality}\label{AsympExp}

Let $\varphi =(\varphi_1,\dots
,\varphi_N)\in\mathcal{C}^{\infty}(\mathbb{K}^N)^N$ (if $
\mathbb{K}=\mathbb{C}$ we understand that
$\varphi_j\in\mathcal{C}^{\infty}(\mathbb{R}^{2N}))$. We assume
that $\varphi_j$ can  take complex values. In particular, we can
consider $\varphi_j\equiv v_j\in\mathbb{C}$ for $j\in\{1,\ldots,N\}$ and then $\varphi
=v\in\mathbb{C}^N$. Define
$$D=D_\varphi =\varphi_1 D_1+\dots
+\varphi_ND_N:\mathcal{C}^\infty( \mathbb{K}^N)\longrightarrow
\mathcal{C}^\infty( \mathbb{K}^N)$$ and put $D^{(k)}=D\circ
\dots\circ D$ $k$-times.

Let us recall a deep identity (cf.
\cite{Milowka},\cite{BMO},\cite{BKMO})
\begin{equation}\label{DI}
(D(f))^k=\frac{1}{k!}\sum \limits _{j=0}^{k}(-1)^j\binom {k}{j}
f^{j}D^{(k)}(f^{k-j}).
\end{equation}

\begin{defin}  Let $q= ||\cdot ||$ be a norm in $
\mathbb{P}(\mathbb{K}^N)$. If $H$ is a homogenous polynomial of
$N$ variables of degree $k\geq 1$ then we consider a differential
operator $D=H(D_1,\dots ,D_N)$ and define
$$m(H,q)=\inf\{s>0:\exists M>0\forall P\in  \mathbb{P}(\mathbb{K}^N)\
||DP||\leq M(\deg P)^s ||P||\}.$$
For $\alpha\in\mathbb{N}^N$ and $H_{\alpha}(x)=x^\alpha$, $x\in\mathbb{K}^N$, we put
$m(\alpha ,q)=m(H_{\alpha},q)$ and $m(q)=\max\limits_{1\leq j\leq
N}m(e_j,q)$. For $k\geq 1$ we put $m_k(q)=\max\{ m(\alpha ,q):\ |\alpha |=k\}$. In particular, $m_1(q)=m(q)$ is {\it Markov's exponent for a norm} $q$. \end{defin}

\begin{rem}
In a special case, if $q(P)=||P||_E$, where $E$ is a compact
subset of $\mathbb{K}^N$, then we define $m(H,E)=m(H,q)$,
$m(\alpha, E)=m(\alpha ,q)$, $m_k(E)=m_k(q)$, $m(E)=m(q)$.
Moreover the last one is Markov's exponent of $E$ which was
recalled in the first section and if $m(E)<\infty$ we say that $E$
has {\it Markov's property}. Let us note the equality (for
subsets of $ \mathbb{R}^N$, cf. \cite{BKMO})
$$m(H_k,E)=km(E),$$ where $H_k(x_1,\dots ,x_N)=x_1^k+\dots +x_N^k$ ($k$ is a fixed
positive even integer).

Since
$$m(\alpha ,q)\leq m(e_1,q)\alpha_1+\dots +m(e_N,q)\alpha_N\leq
\max\limits_{1\leq j \leq N}m(e_j,q)|\alpha |=m(q)|\alpha |,$$ we get the inequality
$$ m_k(q)\leq km(q)\ \Rightarrow\ \frac{1}{k}m_k(q)\leq
m(q).$$\end{rem}
\begin{rem}\label{ThmMil}
From \cite{Milowka} we have $\frac{1}{k}m_k(E)=m(E)$. Therefore
$$\lim\limits_{k\rightarrow\infty} \frac{1}{k}m_k(E)=m(E).$$
\end{rem}
\begin{defin}
Let $q$ be a norm in $
\mathbb{P}(\mathbb{K}^N)$. We define {\it the asymptotic exponent} for $q$,
$$m^*(q):=\limsup\limits_{k\rightarrow\infty} \frac{1}{k}m_k(q).$$
\end{defin}

\begin{rem} Let us note a few basic properties of the above notion.

\begin{itemize}
\item[a)]  If $q_1$ and $q_2$ are two norms on $ \mathbb{P}(
\mathbb{K}^N)$ such that
$$q_1(P) \leq A(\deg P)^a q_2(P),\ q_2(P)\leq B(\deg P)^b q_1(P)
,\ \deg P\geq 1,$$ then $m^*(q_1)=m^*(q_2)$.
\medskip

\item[b)] If $q_1(P)={\rm essup}_\mu |P|$ then
$m^*(q_2)=m(q_1)$.

\item[c)] We have $m^*(q)\leq m(q)$. In general, these exponents do
not need to be equal.
\end{itemize}
\end{rem}
\medskip

Now, we give an example of the norms for
which $m^*(q) < m(q)$. First, we need the following

\begin{pro}\label{EfN0}
For $||\cdot||_0$ a seminorm on $ \mathbb{P}(\mathbb{C})$, $m >0$ and $s\in\mathbb{N}_1$ we define the norm
\[
q_{m,s}(P)=||P||_{m,s}=\sum\limits_{r=0}^\infty \frac{1}{((rs)!)^m} ||P^{(rs)}||_0.
\]

If for every $s\geq 2$ there exist positive constants $A, B$ such that for every $j\in\{1,\ldots,s\}$ and $P\in \mathbb{P}(\mathbb{C})$, $\|P^{(j)}\|_0\leq A \|P\|_0 +B\|P^{(s)}\|_0$,
then  $m_k(q_{m,s})\leq sm\lceil\frac{k}{s}\rceil$ for every $k\in\mathbb{N}_1$.
\end{pro}

\begin{proof} For every $m>0$, $t,s\in\mathbb{N}_1$, $j\in\{1,\ldots,s\}$ and $P\in \mathbb{P}(\mathbb{C})$ we obtain
\begin{align*}
||P^{(st+j)}||_{m,s}=&\sum\limits_{r=0}^\infty \frac{1}{((rs)!)^m}
\|P^{(rs+st+j)}\|_0\\
 \leq & A \sum\limits_{r=0}^\infty
\frac{1}{((rs)!)^m} \|P^{(rs+st)}\|_0\\
&+\max\{B,1\}\sum\limits_{r=0}^\infty \frac{1}{((rs)!)^m}
\|P^{(rs+st+s)}\|_0\\
 \leq & A \sum\limits_{r=0}^{[\frac{\deg
P}{s}]} \frac{1}{((rs)!)^m} \|P^{(rs+st)}\|_0\\
&+\max\{B,1\}\sum\limits_{r=0}^{[\frac{\deg P}{s}]}
\frac{1}{((rs)!)^m} \|P^{(rs+st+s)}\|_0\\
= & A \sum\limits_{r=t}^{[\frac{\deg P}{s}]+t} \frac{1}{(((r-t)s)!)^m} \|P^{(rs)}\|_0\\
&+\max\{B,1\}\sum\limits_{r=t+1}^{[\frac{\deg P}{s}]+t+1}
\frac{1}{(((r-t-1)s)!)^m} \|P^{(rs)}\|_0\\
\leq & (A+\max\{B,1\})(\deg P)^{s(t+1)m}\sum\limits_{r=0}^\infty \frac{1}{((rs)!)^m}
 \|P^{(rs)}\|_0\\ =& (A+\max\{B,1\})(\deg P)^{s(t+1)m}||P||_{m,s}.
\end{align*}
\end{proof}

\begin{exa} Let us consider the norms $q_{m,s}$ defined like in Proposition \ref{EfN0} with seminorm $||P||_0=\sum\limits_{l=0}^{s-1} \frac{1}{l!}|P^{(l)}(0)|$, $s\in\mathbb{N}_1$. Then for every $P\in \mathbb{P}(\mathbb{C})$ and $j\in\{1,\ldots,s-1\}$ we have
\begin{align*}
\|P^{(j)}\|_0=&\sum\limits_{l=0}^{s-1} \frac{1}{l!}|P^{(j+l)}(0)|=\sum\limits_{l=j}^{s-1} \frac{l!}{l!(l-j)!}|P^{(l)}(0)|\\ &+\sum\limits_{l=0}^{j-1} \frac{l!}{l!(s+l-j)!}|P^{(s+l)}(0)|\\
\leq & (s-1)^{s-1}\sum\limits_{l=0}^{s-1} \frac{1}{l!}|P^{(l)}(0)| +\sum\limits_{l=0}^{s-1} \frac{1}{l!}|P^{(s+l)}(0)|\\
\leq & (s-1)^{s-1}\|P\|_0 +\|P^{(s)}\|_0
\end{align*}
From Proposition \ref{EfN0} for $m>0$ and $s\in\mathbb{N}_1$ we obtain $m_k(q_{m,s})\leq sm\lceil\frac{k}{s}\rceil$.

On the other hand for every $m>0$ and $s,n\in\mathbb{N}_1$ we have
$$||x^{sn}||_{m,s}=\sum\limits_{r=0}^\infty \left(
\frac{1}{(rs)!}\right)^m \sum\limits_{l=0}^{s-1} \frac{1}{l!}|(x^{sn})^{(rs+l)}(0)|=\frac{1}{(sn)!^{m-1}}.
$$
and
\begin{align*}
||(x^{sn})^{(st+j)}||_{m,s}=&\sum\limits_{r=0}^\infty \left(
\frac{1}{(rs)!}\right)^m \sum\limits_{l=0}^{s-1} \frac{1}{l!}|(x^{sn})^{(rs+st+j+l)}(0)|\\
=&\frac{(sn)!}{(s-j)!(sn-st-s)!^{m}}.
\end{align*}

Hence for every $k\in\mathbb{N}_1$ we have $m_k(q_{m,s})=sm\lceil\frac{k}{s}\rceil$, where for $x\in\mathbb{R}$, $\lceil
x\rceil$ is the smallest integer greater than or equal to $x$.
From this it follows that $m^*(q_{m,s})=m$ and $m(q_{m,s})=sm$.

\end{exa}

\medskip

Now we formulate main results of this paper.
\medskip

\begin{thm} Let $q$ be a spectral admissible norm for some spectral norm $q_{\sigma}$. Then
$$m^*(q)=\lim\limits_{k\rightarrow\infty} \frac{1}{k}m_k(q)= m(q_{\sigma}).$$
In particular, $m(q_{\sigma})\leq m(q)$.\end{thm}

\begin{cor} Let $q$ be an $E$-admissible norm. Then
$$m^*(q)=\lim\limits_{k\rightarrow\infty} \frac{1}{k}m_k(q)= m(E).$$
In particular, $m(E)\leq m(q)$.\end{cor}

\begin{proof} Firstly, we prove that $m_k(q_{\sigma})=km(q_{\sigma}),\
k\geq 1$.
If for every $j\in\{1,\ldots, N\}$ there exist positive constants $M_j, m_j$ such that for every polynomial $P\in\mathbb{P}_n(\mathbb{K}^N)$,
$$
\|D_jP\|\leq M_j n^{m_j}\|P\|
$$
then for $\alpha\in\mathbb{N}_0^N$ such that $|\alpha|=k$ we have
\begin{align*}
\|D^{(\alpha_1,\ldots,\alpha_N)}P\|\leq& M_1^{\alpha_1}\cdot\ldots\cdot M_N^{\alpha_N} n^{\alpha_1m_1+\ldots+\alpha_Nm_N}\|P\|\\ \leq & (\max_{j\in\{1,\ldots, N\}}M_j)^k n^{km}\|P\|,
\end{align*}
where $m=\max_{j\in\{1,\ldots,N\}}m_j$.
Hence $m_k(q)\leq km(q)$ for every norm $q$.

On the other hand
$$
(D_jP)^k=\frac{1}{k!}\sum_{j=0}^k (-1)^j \binom {k}{j}
P^{j}\frac{\partial^k}{\partial x_j^k}P^{k-j}.
$$
The norm $q_{\sigma}$ is spectral and by the Theorem in \cite{De}
it is submultiplicative. Hence, if an $\varepsilon >0$ is fixed,
\begin{align*}
\|(D_jP)\|_{\sigma}^k\leq & const. (\varepsilon )
\frac{1}{k!}\sum_{j=0}^k \binom {k}{j}
\|P\|_{\sigma}^{j}(n(k-j))^{m_k(q_{\sigma})+\varepsilon}\|P\|_{\sigma}^{k-j}\\
\leq & const.(\varepsilon )\frac{2^k}{k!}(nk)^{m_k(q_\sigma
)+\varepsilon}||P||_\sigma^k,
\end{align*}
which shows that $m(q_\sigma )\leq m_k(q_\sigma )/k+\varepsilon
/k$. Letting $\varepsilon\rightarrow 0+$ we get the inequality
$m(q_\sigma )\leq m_k(q_\sigma )/k$ and finally
$m_k(q_{\sigma})=km(q_{\sigma}),\ k\geq 1$.
\medskip

Now, let $s>m(\alpha ,q_\sigma)$. Then
\begin{align*}
||D^\alpha P||\leq & B(\deg P)^b||D^\alpha P||_\sigma\leq
BM_s(\deg
P)^{b+s}||P||_\sigma\\
\leq & BM_sA(\deg P)^{b+a+s}||P||,
\end{align*} which gives
$$m(\alpha ,q)\leq b+a+s\ \Rightarrow\ m(\alpha ,q)\leq
b+a+m(\alpha,q_\sigma)
$$
and therefore $ m_k(q)\leq b+a+ m_k(q_\sigma)=b+a+km(q_\sigma).$
Hence
$$m^*(q)=\limsup\limits_{k\rightarrow\infty}\frac{1}{k}m_k(q)\leq m(q_\sigma).$$
Analogously, let $s>m(\alpha ,q)$. Then
\begin{align*}
||D^\alpha P||_\sigma\leq & A(\deg P)^{a}||D^\alpha P||\leq
AM_s'(\deg
P)^{a+s}||P||\\
\leq & ABM_s'(\deg P)^{a+b+s}||P||_\sigma,
\end{align*} which implies
$m(\alpha ,q_\sigma )\leq a+b+s$ and $m(\alpha ,q_\sigma)\leq
a+b+m(\alpha,q),$ whence $ km(q_\sigma)=m_k(q_\sigma)\leq
a+b+m_k(q).$ Hence
$$m(q_\sigma)\leq \liminf\limits_{k\rightarrow\infty}\frac{1}{k}m_k(q)\leq
\limsup\limits_{k\rightarrow\infty}\frac{1}{k}m_k(q)\leq
m(q_\sigma).$$
\end{proof}
\bigskip

The second statement in the following important corollary is very
useful. Also the third statement gives new result.

\begin{cor}
\begin{itemize}
\item[a)] If a norm $q$ has GNP with the spectral norm $q_\sigma$ then
\[m^*(q)=m(q)\ \ \Leftrightarrow\ \ m(q)=m(q_\sigma).\]
\item[b)] If for a norm $q=||\cdot ||$ we have Markov's inequality
$$||D_jP||\leq M(\deg P)^{m(q_\sigma )}||P||,\ j=1,\dots,N$$ then
the exponent $m(q_\sigma )$ is the best possible. In particular,
$m(q)=m(q_\sigma )$.
\item[c)] If $E$ is an UPC subset of $\mathbb{R}^N$, then
$m_p(E)\geq m(E)$, where $m_p(E)$ is Markov's exponent with
respect to the Lebesgue measure.
\end{itemize}
\end{cor}

\begin{rem}
In papers where Markov's inequality in $L^p$ norms was proved with
the best possible exponent, usually it was difficult and
time-consuming to prove the optimality of the exponent, which is
Mar\-kov's exponent for such kind of norms (cf.
\cite{HSzT},\cite{Goe1987},\cite{DaugavetRafalson},\cite{Konjagin78}). By applying the above corollary it is done automatically.

Let us consider another (simple) example. By Bernstein's
inequality $$||\sqrt{1-x^2}P'(x)||_{[-1,1]}\leq (\deg
P)||P||_{[-1,1]}$$ and by Schur's inequality
$$||P||_{[-1,1]}\leq (\deg P+1)||\sqrt{1-x^2}P(x)||_{[-1,1]}$$ we
get Markov's inequality with respect to Schur's norm
$$ ||\sqrt{1-x^2}P'(x)||_{[-1,1]}\leq \deg P(\deg P+1)
||\sqrt{1-x^2}P(x)||_{[-1,1]},$$ with exponent 2, which is, by the
corollary  above, the best possible.

\end{rem}

\section{Markov's exponent for a sequence of polynomials in $\mathbb{P}( \mathbb{C})$.}\label{MarExpSeq}

\begin{defin} Fix a compact set $E\subset\mathbb{C}$ and a sequence
of polynomials $\widehat{\mathcal{P}}=(\widehat{P}_n)_{n\geq
0}\subset \mathbb{P}( \mathbb{C}),\ \deg \widehat{P}_n=n$. Put,
for $k\geq 1$,
$$m_k( \widehat{\mathcal{P}}):=\limsup\limits_{n\rightarrow\infty} \frac{\log (
||\widehat{P}^{(k)}_n||_E/||\widehat{P}_n||_E)}{\log n}$$ and
$m^*( \widehat{\mathcal{P}}):=\limsup\limits_{k\rightarrow\infty}
\frac{1}{k}m_k( \widehat{\mathcal{P}})$.
\end{defin}

\begin{thm} Let $\widehat{\mathcal{P}}=(\widehat{P}_n)_{n\geq
0}$ be an orthonormal system (with respect to a~probabilistic
measure $\mu$ supported on $E$) such that
\[\limsup\limits_{n\rightarrow\infty}\frac{\log
||\widehat{P}_n||_E}{\log n}=\alpha <\infty.\] Then $m(E)=m^*( \widehat{\mathcal{P}})$.

\end{thm}

\begin{proof} It is clear that $m^*( \widehat{\mathcal{P}})\leq
m(E)$. Assume that $m^*( \widehat{\mathcal{P}})=:\gamma <\infty$.
Fix an $\varepsilon >0$. There exists $k_0$ such that for all
$k\geq k_0$, $m_k(\widehat{\mathcal{P}})\leq k(\gamma +\varepsilon
)$. If $k\geq k_0$ is fixed then for $n\geq n_0$ we have an estimation $$||\widehat{P}^{(k)}_n||_E\leq n^{m_k(
\widehat{\mathcal{P}})+\varepsilon} ||\widehat{P}_n||_E \leq
An^{k(\gamma +\varepsilon )+\alpha +2\varepsilon},$$ where $A$ is
a positive constant.  Let $\langle ,\rangle$ be a scalar product
in $L^2(\mu )$. It is well known that for $P\in \mathbb{P}_n(
\mathbb{C})$ we have
$$P=\sum_{j=0}^n\langle P,\widehat{P}_j\rangle \widehat{P}_j.$$
Hence
\begin{align*}
 ||P^{(k)}||_E\leq &\sum\limits_{j=0}^n|\langle
P,\widehat{P}_j\rangle| ||\widehat{P}_j^{(k)}||_E
\leq ||P||_E\sum\limits_{j=0}^n||\widehat{P}_j^{(k)}||_E\\ &\leq
AB_kn^{k(\gamma +\varepsilon ) +\alpha +1+2\varepsilon} ||P||_E.
\end{align*}
Here $B_k$ is a positive constant. Now we can write
$$m_k(E)/k\leq (\gamma +\varepsilon) +(\alpha +1+2\varepsilon
)/k,$$ which gives
$$m(E)=m^*(E)=\lim\limits_{k\rightarrow\infty}m_k(E)/k\leq
\gamma +\varepsilon .$$ Since $\varepsilon >0$ was arbitrary, we
get $m(E)\leq\gamma$ which finishes the proof.
\end{proof}


\begin{exa} $m([-1,1])=m_1( \widehat{\mathcal{P}}_{\alpha
,\beta}),$ where $\widehat{\mathcal{P}}_{\alpha ,\beta}$ is the
family of normalized Jacobi polynomials.

\end{exa}



{\bf Acknowledgment.} The authors were partially supported by the
NCN grant No. 2013/11/B/ST1/03693. An important part of this paper
was written during authors' visit at University of Padova. We are grateful to the Department of Mathematics for their helpful support and for creating an intellectual atmosphere.

%
%



\begin{thebibliography}{}
%
%

\bibitem{AleksovNikolov17} D. Aleksov, G. Nikolov, \emph{Markov  L$_2$ inequality  with the
Gegenbauer weight}, arXiv:702.05963v1 [math. CA] 20 Feb 2017.

\bibitem{AleksovNikolovShadrin16} D. Aleksov, G. Nikolov, A.
Shadrin, \emph{On the Markov inequality in the L$_2$ norm with the
Gegenbauer weight}, J. Approx. Ttheory {\bf 208}, 9--20 (2016)


\bibitem{Baran1} M. Baran, \emph{Bernstein type theorems  for compact sets in
$\mathbb{R}^{n}$}, J. Approx. Theory {\bf 69}(2),
156--166 (1992)


\bibitem{Baran3} M. Baran, \emph{Markov inequality on sets with polynomial parametrization}, Ann. Polon. Math. {\bf 60}(1), 69--79  (1994)

\bibitem{Baran4} M. Baran, \emph{Complex equilibrium  measure  and Bernstein type theorems for compact sets in $\mathbb{R}^{n}$},  Proc.  Amer. Math. Soc. {\bf 123}(2), 485--494 (1995)

\bibitem{Baran5} M. Baran \emph{New approach to Markov inequality in $L^{p}$ norms}, in: Approximation theory. In memory of A.K. Varma, I. Govil et al. -- editors, M. Dekker, Inc., New York--Basel--Hong Kong, 75--85 (1998)

\bibitem{BBCM} M. Baran, L. Bia\l{}as-Cie\.z, B. Mil\'owka,\emph{ On the best exponent in Markov's inequality}, Potential Anal. {\bf 38}, 635--651 (2013)


\bibitem{BKMO} M. Baran, A. Kowalska, B. Mil\'{o}wka, P. Ozorka \emph{Identities for a derivation operator and their applications}, Dolomites Research Notes on Approximation {\bf 8}, 102--110 (2015)

\bibitem{BMO} M. Baran, B. Mil\'{o}wka, P.Ozorka \emph{Markov's property for kth derivative}, Ann. Pol. Math. {\bf 106}, 31--40 (2012)

\bibitem{BaranPlesniak}  M. Baran, W. Ple\'sniak, \emph{ Markov's exponent of compact sets in $\mathbb{C}^n$}, Proc. Amer. Math. Soc.  {\bf 123}, 2785--2791 (1995)

\bibitem{LBC99} L. Bia\l{}as-Cie\.z, \emph{Equivalence of Markov's and Schur's inequalities on compact subsets of the complex plane}, J. Inequal.  Appl. {\bf 3}(1), 45--49 (1999)

\bibitem{LBCSroka} L. Bia\l{}as-Cie\.z, G. Sroka, \emph{Polynomial inequalities in $L^p$ norms
with generalized Jacobi weights} (2017) (preprint)

\bibitem{BottcherDorfler09} A. B\"ottcher, P. D\"orfler, \emph{On
the best constants in inequalities of the Markov and Wirtinger
types for polynomials on the half-line}, Linear Algebra Appl. {\bf
430}(4), 1057--1069 (2009)

\bibitem{BottcherDorfler10} A. B\"ottcher, P. D\"orfler,
\emph{Weighted Markov-type inequalities norms of Volterra
operators, and zeros of Bessel functions}, Math. Nachr. {\bf
283}(1), 40--57 (2010)

\bibitem{BottcherDorfler11} A. B\"ottcher, P. D\"orfler,
\emph{Inequalities of the Markov type for partial derivatives of
polynomials in several variables}, J. Integral Equations Appl.
{\bf 23}(1), 1--37 (2011)

\bibitem{BorweinErdelyi} P. Borwien, T. Erd\'elyi,
\emph{Polynomials and Polynomial Inequalities}, Springer, New York (1995)

\bibitem{DaugavetRafalson} I.K. Daugavet, S.Z. Rafalson,
\emph{Some inequalities of Markov-Nikolskii type for algebraic
polynomials}, Vestnik Leningrad. Univ. Math. Mekh. Astronom. {\bf
1}, 15--25 (1972) (in Russian)

\bibitem{De} H.V.Dedania, \emph{A seminorm with square property is automatically submultiplicative}, Proc. Indian. Acad. Sci ( Math. Sci.), {\bf 108}, 51-53 (1998)

\bibitem{Dorfler87} P. D\"orfler, \emph{New inequalities of Markov
type}, SIAM J. Math. Anal. {\bf 18}(2), 490--494 (1987)

\bibitem{Dorfler02} P. D\"orfler, \emph{Asymptotics of the Best
Constant in a Certain Markov-Type Inequality}, J. Approx. Theory
{\bf 114}, 84--97 (2002)

\bibitem{Goe1979} P. Goetgheluck, \emph{Polynomial inequalities and Markov's inequality in weighted $L^p$ spaces}, Acta Math. Acad. Sci. Hungar. {\bf 33}(3-4), 325--331 (1979)

\bibitem{Goe1987} P. Goetgheluck, \emph{On Markov's inequality on locally Lipshitzian compact subsets of $\mathbb{R}^N$ in $L^{p}$-Spaces}, J. Approx. Theory {\bf 49}, 303--310 (1987)

\bibitem{Goe1989} P. Goetgheluck, \emph{Polynomial
Inequalities on General Subsets of $\mathbb{R}^{N}$}, Coll. Math. {\bf 57}(1), 127--136 (1989)

\bibitem{Goe1990} P. Goetgheluck, \emph{On the Markov Inequality in $L^p$-spaces}, J. Approx. Theory {\bf 62}(2), 197--205 (1990)

\bibitem{Goe1994} P. Goetgheluck, \emph{On the problem of Sharp Exponents in Multivariate Nikolski-Type Inequality}, J. Approx. Theory {\bf 77}, 167--178 (1994)

\bibitem{Goe1998} P. Goetgheluck, \emph{Two polynomial division inequalities in $L^p$},
 J. Inequal.  Appl. {\bf 2}(3), 285--296 (1998)

\bibitem{Gon2014} A. Goncharov, \emph{Best exponents in Markov's inequalities}, Math. Inequal. Appl. {\bf 17}, 1515–-1527 (2014)

\bibitem{Govil} N.K. Govil, R.N. Mohapatra, \emph{Markov and Bernstein type inequalities for polynomials}, J. Inequal. Appl. {\bf 3}, 349--387 (1999)

 \bibitem{HSzT} E. Hille, G. Szeg\"{o}, J. Tamarkin, \emph{On some generalisation of a theorem of A.~Markoff}, Duke Math. J {\bf 3}, 729--739 (1937)

\bibitem{Jonsson2} A. Jonsson, \emph{Markov's inequality and zeros of orthogonal polynomials on fractal sets}, J. Approx. Theory {\bf 78}(1), 87--97 (1994)

\bibitem{Jonsson3} A. Jonsson, \emph{Measures satisfying a refined doubling condition and absolute continuity}, Proc. Amer. Math. Soc. {\bf 123}(8), 2441--2446 (1995)


\bibitem{Konjagin78} S.V. Konjagin, \emph{Estimation of the derivatives of polynomials}, Dokl. Akad. Nauk SSSR {\bf 243}, 1116--1128 (1978) (in Russian)

\bibitem{Milovanovic} G.V Milovanovi\'c, D.S. Mitrinovi\'c, T.M. Rassias, \emph{Topics in polynomials: extremal problems, inequalities, zeros}, World Scietific Publishing, River Edge, NJ, (1994)

\bibitem{Milowka} B. Mil\'{o}wka,
\emph{Markov's inequality and a generalized Ple\'sniak condition}, East Jour. of Approx. {\bf 11}(3), 291--300 (2005).

\bibitem{Nikolski} S.M. Nikolskii, \emph{Inequalities for entire functions of finite degree and their application in the theory of differentiable functions of several variables}, Trudy Math. Inst. Stieklov {\bf 38}, 244--278 (1960) (in Russian)


\bibitem{PaPl86} W. Paw\l{}ucki, W. Ple\'sniak, \emph{Markov's inequality and $\mathcal{C}^\infty$ functions on sets with cusps}, Math. Ann. {\bf 275}, 467--480 (1986)

\bibitem{PaPl88} W. Paw\l{}ucki, W. Ple\'sniak, \emph{Extension of $C^\infty$ functions from sets with polynomial cusps}, Studia Math. {\bf 88}, 279--287 (1988)


\bibitem{Pierzchala4} R. Pierzcha\l{}a, \emph{Remez type inequality on sets with cusps}, Adv. Math. {\bf 281}, 508--552 (2015)

\bibitem{Pl90} W. Ple\'sniak, \emph{Markov's inequality and the existence of an extension operator for $C^\infty$ functions}, J. Approx. Theory {\bf 61}(1), 106--117 (1990)

\bibitem{Pl98} W. Ple\'sniak, \emph{Recent progress in multivariate Markov inequality}, Approximation theory, Monogr. Textbooks Pure Appl. Math., Dekker, New York, 449-464 (1998)

\bibitem{Pl06} W. Ple\'sniak, \emph{In\'egalit\'e de Markov en plusieurs variables}, Int. J. Math. {\bf 14}, 1--12 (2006)

\bibitem{RahmanSchmeisser} Q.I. Rachman, G. Schmeisser,
\emph{Analytic Theory of Polynomials}, Oxford University Press,
Oxford (2002)

\bibitem{Sroka} G. Sroka, \emph{Constants in V.A. Markov's
inequality in L$^p$ norms}, J. Approx. Theory {\bf 194},
27--34 (2015)

\bibitem{Timan} A.F. Timan, \emph{The  theory of approximation of functions of a real variable}, Pergamon Press Book, New York (1963)

\bibitem{Totik1} V. Totik, \emph{Polynomial inverse images and polynomial inequalities}, Acta Math. {\bf 187}, 139--160 (2001)

\bibitem{Totik2} V. Totik, \emph{How to prove results for
polynomials on several intervals?}, Approximation theory, DARBA, Sofia, 397--410 (2002)

\bibitem{Totik3} V. Totik, \emph{Equiulibrium measures and
polynomials}, European Congress of Mathematics, Eur. Math. Soc., Z\"urich, 501--514 (2005)


\bibitem{ToVa} V. Totik, T. Varga \emph{Chebyshev and fast decreasing polynomials}, Proc. Lond. Math. Soc.  {\bf 110}(5), 1057-1098 (2015)

\bibitem{Zer} A. Z\'{e}riahi, \emph{Inegalit\'{e}s de Markov et
d\'{e}veloppment en s\'{e}rie de polynomes orthogonaux des fonctions $C^{\infty}$ et $A^{\infty}$}, in: Proc. Special Year of Complex Analysis of the Mittag-Leffler Institute 1987-88, J. F. Fornaess (ed.), Princeton Univ. Press, Princeton, NJ, 683-701 (1993)

\bibitem{Zel} W. \.{Z}elazko, \emph{Banach Algebras}, PWN-Elsevier, Warszawa (1973)


\end{thebibliography}


\vspace{1cm}
M. Baran \\
Department of Applied Mathematics, University of Agriculture  in Kra\-k\'ow, Balicka 253C,  30-198 Krak\'ow, Poland \\
\email{mbaran@ar.krakow.pl}\\
A. Kowalska\\ 
Institute of Mathematics, Pedagogical  University, Podchor\c{a}\.zych 2,  30-084 Krak\'ow, Poland\\
\email{kowalska@up.krakow.pl}

\end{document}